\documentclass[smallextended]{svjour3}     % onecolumn (second format)
\smartqed  % flush right qed marks, e.g. at end of proof
\usepackage{graphicx}
\usepackage{amsfonts}
\usepackage{amssymb}
\usepackage{amsmath}
\usepackage{amsxtra}
\usepackage{latexsym}
\begin{document}

\title{\bf  \Large Inexact Newton regularization methods in Hilbert scales
}

\titlerunning{Short form of title}        % if too long for running head

\author{Qinian Jin         \and
    Ulrich Tautenhahn %etc.
}

\authorrunning{Short form of author list} % if too long for running head

\institute{Qinian Jin \at
 Department of Mathematics, Virginia Tech, Blacksburg, VA 24060, USA\\
\email{qnjin@math.vt.edu}
%              Tel.: +123-45-678910\\
%              Fax: +123-45-678910\\
           \and
Ulrich Tautenhahn \at
Department of Mathematics, University of Applied Sciences Zittau/G\"{o}rlitz, PO Box 1454,
02754 Zittau, Germany\\
          \email{u.tautenhahn@hs-zigr.de}
}


\newtheorem{Assumption}{Assumption}

\def\A{\mathcal A}
\def\B{\mathcal B}

\def\d{\delta}
\def\ds{\displaystyle}
\def\e{{\epsilon}}
\def\eb{\bar{\eta}}
\def\enorm#1{\|#1\|_2}
\def\Fp{F^\prime}
\def\fishpack{{FISHPACK}}
\def\fortran{{FORTRAN}}
\def\gmres{{GMRES}}
\def\gmresm{{\rm GMRES($m$)}}
\def\Kc{{\cal K}}
\def\norm#1{\|#1\|}
\def\wb{{\bar w}}
\def\zb{{\bar z}}

\def\B{\mathcal B}
\def\A{\mathcal A}
\def\a{\alpha}
\def\b{\beta}
\def\d{\delta}
\def\la{\lambda}

\maketitle

\begin{abstract}
We consider a class of inexact Newton regularization methods for solving nonlinear
inverse problems in Hilbert scales. Under certain conditions we obtain the order
optimal convergence rate result.

%\keywords{}

%\subclass{65J15 \and 65J20 \and 47H17}
\end{abstract}

%\catcode`@=11 \@addtoreset{equation}{chapter}
\def\theequation{\thesection.\arabic{equation}}
\catcode`@=12

\section{\bf Introduction}
\setcounter{equation}{0}

In this paper we consider the nonlinear inverse problems
\begin{equation}\label{1}
F(x)=y,
\end{equation}
where $F: D(F)\subset
X\mapsto Y$ is a nonlinear Fr\'{e}chet differentiable operator
between two Hilbert spaces $X$ and $Y$ whose norms and inner
products are denoted as $\|\cdot\|$ and $(\cdot, \cdot)$
respectively. We assume that (\ref{1}) has a solution $x^\dag$
in the domain $D(F)$ of $F$, i.e. $F(x^\dag)=y$. We use
$F'(x)$ to denote the Fr\'{e}chet derivative of $F$ at $x\in D(F)$
and $F'(x)^*$ the adjoint of $F'(x)$.
A characteristic property of such problems is
their ill-posedness in the sense that their solutions do not
depend continuously on the data. Let $y^\d$ be the only available
approximation of $y$ satisfying
\begin{equation}\label{1.2}
\|y^\delta-y\|\le \delta
\end{equation}
with a given small noise level $\delta> 0$. Due to the
ill-posedness, the regularization techniques should be employed to
produce from $y^\d$ a stable approximate solution of (\ref{1}).

Many regularization methods have been considered in the last two
decades. In particular, the nonlinear Landweber iteration \cite{HNS96},
the Levenberg-Marquardt method \cite{H97,Jin10a},
and the exponential Euler iteration \cite{HHO09} have been applied to
solve nonlinear inverse problems. These methods take the form
\begin{equation}\label{m1}
x_{n+1} =x_n-g_{\alpha_n} \left(F'(x_n)^*F'(x_n)\right) F'(x_n)^*
\left(F(x_n)-y^\d\right),
\end{equation}
where $x_0$ is an initial guess of $x^\dag$, $\{\a_n\}$ is a sequence of
positive numbers, and $\{g_\a\}$ is a family of spectral filter functions.
The scheme (\ref{m1}) can be derived by applying the linear regularization
method defined by $\{g_\a\}$ to the equation
\begin{equation}\label{n1}
F'(x_n)(x-x_n)=y^\delta-F(x_n).
\end{equation}
which follows from (\ref{1}) by replacing $y$ by $y^\d$ and $F(x)$ by
its linearization $F(x_n)+F'(x_n)(x-x_n)$ at $x_n$. It is easy to see that
$$
F(x_n)-y^\d+F'(x_n)(x_{n+1}-x_n)=r_{\a_n}(F'(x_n)F'(x_n)^*) (F(x_n)-y^\d),
$$
where
\begin{equation}\label{r1}
r_{\a}(\lambda)=1-\lambda g_\a(\lambda)
\end{equation}
which is called the residual function associated with $g_\a$.
For well-posed problems where $F'(x_n)$ is invertible, usually one has
$\|r_{\a_n}(F'(x_n)F'(x_n)^*)\|\le \mu_n<1$ and consequently
\begin{equation}\label{m2}
\|F(x_n)-y^\d+F'(x_n)(x_{n+1}-x_n)\|\le \mu_n \|F(x_n)-y^\d\|.
\end{equation}
Thus the methods belong to the class of inexact Newton methods \cite{DES82}.
For ill-posed problems, however, there only holds $\|r_{\a_n}(F'(x_n)F'(x_n)^*)\|\le 1$
in general. In \cite{H97} the Levenberg-Marquardt scheme was considered
with $\{\a_n\}$ chosen adaptively so that (\ref{m2}) holds and
the discrepancy principle was used to terminate the iteration. The order optimal
convergence rates were derived recently in \cite{H2010}.
The general methods (\ref{m1}) with $\{\a_n\}$ chosen adaptively to satisfy
(\ref{m2}) were considered later in \cite{R99,LR2010}, but only suboptimal
convergence rates were derived in \cite{R01} and the convergence analysis
is far from complete. On the other hand, one may consider the method
(\ref{m1}) with $\{\a_n\}$ given a priori. This has been done for the Levenberg-Marquardt
method in \cite{Jin10a} and the exponential Euler method in \cite{HHO09} for instance.

In this paper we will consider the inexact Newton methods in
Hilbert scales which are more general than (\ref{m1}). Let
$L$ be a densely defined self-adjoint strictly positive linear
operator in $X$.  For each $r\in {\mathbb R}$,
we define $X_r$ to be the completion of $\cap_{k=0}^\infty D(L^k)$
with respect to the Hilbert space norm
$$
\|x\|_r:= \|L^r x\|.
$$
This family of Hilbert spaces $(X_r)_{r\in {\mathbb R}}$ is called the
Hilbert scales generated by $L$. Let $x_0\in D(F)$ be an initial guess of $x^\dag$.
The inexact Newton method in Hilbert scales defines the iterates
$\{x_n\}$ by
\begin{equation}\label{2}
x_{n+1}=x_n- g_{\a_n}\left(L^{-2s}F'(x_n)^* F'(x_n)\right) L^{-2s} F'(x_n)^* (F(x_n)-y^\d),
\end{equation}
where $s\in {\mathbb R}$ is a given number to be specified later, and
$\{\a_n\}$ is an a  priori given sequence of positive numbers with suitable properties.
We will terminate the iteration by the discrepancy principle
\begin{equation}\label{DP}
\|F(x_{n_\d})-y^\d\|\le \tau \d <\|F(x_n)-y^\d\|, \quad 0\le n<n_\d
\end{equation}
with a given number $\tau>1$ and consider the approximation property
of $x_{n_\d}$ to $x^\dag$ as $\d\rightarrow 0$. We will establish for a large class of
spectral filter functions $\{g_\a\}$ the order optimal convergence rates
for the method defined by (\ref{2}) and (\ref{DP}).

Regularization in Hilbert scales has been introduced in \cite{Na1984}
for the linear Tikhonov regularization with the major aim to prevent
the saturation effect. Such technique has been extended in various ways,
in particular, a general class of regularization methods in Hilbert scales has been
considered in \cite{Tau1996} with the regularization parameter chosen
by the Morozov's discrepancy principle. Regularization in Hilbert scales
have also been applied for solving nonlinear ill-posed problems. The nonlinear
Tikhonov regularization in Hilbert scales has been considered in
\cite{KT95,EHN96}, a general continuous regularization scheme for nonlinear
problems in Hilbert scales has been considered in \cite{Tau98}, the general
iteratively regularized Gauss-Newton methods in Hilbert scales has been considered in
\cite{Jin2000}, and the nonlinear Landweber iteration in Hilbert scales
has been considered in \cite{N2000}.

This paper is organized as follows. In Section 2 we first briefly
review the relevant properties of Hilbert scales, and then formulate
the necessary condition on $\{\a_n\}$, $\{g_\a\}$ and $F$ together
with some crucial consequences. In Section 3 we obtain the main
result concerning the order optimal convergence property of the
method given by (\ref{2}) and (\ref{DP}). Finally we present in
Section 4 several examples of the method (\ref{2}) for which
$\{g_\a\}$ satisfies the technical conditions in Section 2.

\section{\bf Assumptions}
\setcounter{equation}{0}

We first briefly review the relevant properties of the Hilbert scales
$(X_r)_{r\in {\mathbb R}}$ generated by a densely defined self-adjoint
strictly positive linear operator $L$ in $X$, see \cite{EHN96}.
It is well known that $X_r$ is densely
and continuously embedded into $X_q$ for any $-\infty<q<r<\infty$, i.e.
\begin{equation}\label{1000}
\|x\|_q\le \theta^{r-q}\|x\|_r, \quad x\in X_r,
\end{equation}
where $\theta>0$ is a constant such that
\begin{equation}\label{1111}
\|x\|^2\le \theta (L x, x), \quad x\in D(L).
\end{equation}
Moreover there holds the important interpolation
inequality, i.e. for any $-\infty<p<q<r<\infty$ there holds for any $x\in X_r$ that
\begin{equation}\label{2.1}
\|x\|_q\le \|x\|_p^{\frac{r-q}{r-p}}\|x\|_r^{\frac{q-p}{r-p}}.
\end{equation}
Let $T:X\mapsto Y$ be a bounded linear operator satisfying
$$
m\|h\|_{-a}\le \|Th\|\le M\|h\|_{-a}, \quad h\in X
$$
for some constants $M\ge m>0$ and  $a\ge 0$. Then the operator
$A:=TL^{-s}:X\mapsto Y$ is bounded for $s\ge -a$ and the adjoint
of $A$ is given by $A^*=L^{-s}T^*$, where $T^*:Y\mapsto X$ is the
adjoint of $T$. Moreover, for any $|\nu|\le 1$ there hold
\begin{equation}\label{2.5}
R((A^*A)^{\nu/2})=X_{\nu(a+s)}
\end{equation}
and
\begin{equation}\label{2.3}
 \underline{c}(\nu) \|h\|_{-\nu(a+s)}\le\|(A^*A)^{\nu/2}h\| \le
\overline{c}(\nu) \|h\|_{-\nu(a+s)}
\end{equation}
on $D((A^*A)^{\nu/2})$, where
$$
\underline{c}(\nu):=\min\{m^\nu,M^\nu\} \quad \mbox{and} \quad
\overline{c}(\nu)=\max\{m^\nu, M^\nu\}.
$$
If $g:[0,\|A\|^2]\mapsto {\mathbb R}$ is a continuous function, then
\begin{equation}\label{2.2}
g(A^*A)L^s=L^sg(L^{-2s}T^*T).
\end{equation}

In order to carry out the convergence analysis on the method
defined by (\ref{2}) and (\ref{DP}), we need to impose
suitable conditions on $\{\a_n\}$, $\{g_\a\}$ and $F$.
For the sequence $\{\a_n\}$ of positive numbers, we set
\begin{equation}\label{a1}
s_{-1}=0, \qquad s_n:=\sum_{j=0}^n \frac{1}{\a_j}, \qquad n=0, 1,\cdots.
\end{equation}
We will assume that there are constants $c_0>1$ and $c_1>0$ such that
\begin{equation}\label{60}
\lim_{n \rightarrow \infty} s_n =\infty, \quad
s_{n+1}\le c_0 s_n \quad \mbox{and} \quad 0<\a_n \le c_1, \quad n=0, 1,\cdots.
\end{equation}
We will also assume that, for each $\a>0$, the function $g_\a$ is defined
on $[0,1]$ and satisfies the following structure condition,
where ${\mathbb C}$ denotes the complex plane.

\begin{Assumption}\label{A4}
For each $\a>0$, the function
$$
\varphi_\a(\lambda):=g_\a(\lambda)-\frac{1}{\a+\lambda}
$$
extends to a complex analytic
function defined on a domain $D_\a\subset {\mathbb C}$ such that
$[0, 1]\subset D_\a$, and there is a contour $\Gamma_\a\subset
D_\a$ enclosing $[0, 1]$ such that
\begin{equation}\label{2.6}
|z|\ge \frac{1}{2}\a \quad \mbox{and} \quad \frac{|z|+\la}{|z-\la|}\le b_0, \qquad \forall z \in
\Gamma_\a, \, \a>0 \mbox{ and } \la \in [0,1],
\end{equation}
where $b_0$ is a constant independent of $\a>0$.
Moreover, there is a constant $b_1$ such that
\begin{equation}\label{2.8}
\int_{\Gamma_\a} \left|\varphi_\a(z)\right| |d z|\le b_1
\end{equation}
for all $0<\a\le c_1$.
\end{Assumption}

By using the spectral integrals for self-adjoint operators, it follows easily from
(\ref{2.6}) in Assumption \ref{A4} that for any bounded
linear operator $A$ with $\|A\|\le 1$ there holds
\begin{equation}\label{200}
\|(z I-A^*A)^{-1}(A^*A)^\nu\|\le \frac{b_0}{|z|^{1-\nu}}
\end{equation}
for $z \in \Gamma_\a$ and $0\le \nu\le 1$.

Moreover, since Assumption \ref{A4} implies $\varphi_\a(z)$ is analytic
in $D_\a$ for each $\a>0$, there holds the Riesz-Dunford formula (see \cite{BK04})
$$
\varphi_\a(A^*A)=\frac{1}{2\pi i} \int_{\Gamma_\a} \varphi_\a(z) (z
I-A^*A)^{-1} d z
$$
for any linear operator $A$ satisfying $\|A\|\le 1$.

\begin{Assumption}\label{A5}
Let $\{\a_n\}$ be a sequence of positive numbers,
let $\{s_n\}$ be defined by (\ref{a1}). There is a constant $b_2>0$ such that
\begin{align}
0\le \la^\nu \prod_{k=j}^n r_{\a_k}(\la)&\le (s_n-s_{j-1})^{-\nu},  \label{g1}\\
0\le \la^\nu g_{\a_j}(\la) \prod_{k=j+1}^n r_{\a_k}(\la)&\le b_2
\frac{1}{\a_j} (s_n-s_{j-1})^{-\nu} \label{g2}
\end{align}
for $0\le \nu\le 1$,  $0\le \la \le 1$ and $j=0, 1, \cdots, n$, where
$r_\a(\lambda)$ is defined by (\ref{r1}).
\end{Assumption}

In Section \ref{Sect4} we will give several important examples of
$\{g_\a\}$ satisfying Assumptions \ref{A4} and \ref{A5}. These
examples of $\{g_\a\}$ include the ones arising from (iterated)
Tikhonov regularization, asymptotical regularization, Landweber
iteration and Lardy method.

\begin{lemma}\label{L20}
The inequality (\ref{g1}) implies for $0\le \nu\le 1$ and $\a>0$ that
\begin{equation}
0\le \la^{\nu} (\a+\la)^{-1} \prod_{k=j+1}^n r_{\a_k}(\la)
\le 2 \a^{\nu-1} \left(1+\a (s_n-s_j)\right)^{-\nu} \label{g3}
\end{equation}
for all $0\le \la\le 1$ and $j=0, 1, \cdots, n$.
\end{lemma}

\begin{proof}
For $0\le \nu\le 1$ and $\a>0$ it follows from (\ref{g1}) that
\begin{align*}
0\le \la^{\nu} (\a+\la)^{-1} \prod_{k=j+1}^n r_{\a_k}(\la)
& \le \min\left\{ \a^{\nu-1}, \a^{-1} (s_n-s_j)^{-\nu}\right\}\\
&= \a^{\nu-1} \min\left\{1, \a^{-\nu} (s_n-s_j)^{-\nu}\right\}\\
&\le 2^\nu  \a^{\nu-1} \left(1+\a (s_n-s_j)\right)^{-\nu} \label{g3}
\end{align*}
for all $0\le \la\le 1$ and $j=0, 1, \cdots, n$. \hfill $\Box$
\end{proof}

\begin{Assumption}\label{A1}
(a) There exist constants $a\ge 0$ and $0<m\le M<\infty$ such that
$$
m\|h\|_{-a}\le \|F'(x) h\|\le M \|h\|_{-a}, \quad h\in X
$$
for all $x\in B_\rho(x^\dag)$.

(b) $F$ is properly scaled so that $\|F'(x) L^{-s}\|_{X\to Y}
\le \min\{1,\sqrt{\a_0}\}$ for all $x\in B_\rho(x^\dag)$, where $s\ge -a$.

(c) There exist $0<\beta\le 1$, $0\le b\le a$ and $K_0\ge 0$ such that
\begin{equation}\label{4}
\|F'(x)^*-F'(x^\dag)^*\|_{Y\to X_b}\le K_0\|x-x^\dag\|^\beta
\end{equation}
for all $x\in B_\rho(x^\dag)$.

\end{Assumption}

The number $a$ in condition (a) can be interpreted as the
degree of ill-posedness of $F'(x)$ for $x\in B_\rho(x^\dag)$.
When $F$ satisfies the condition
\begin{equation}\label{A10}
F'(x)=R_x F'(x^\dag) \quad \mbox{and} \quad \|I-R_x\|\le K_0\|x-x^\dag\|,
\end{equation}
which has been verified in \cite{HNS96} for several nonlinear inverse problems,
condition (a) is equivalent to
$$
m\|h\|_{-a}\le \|F'(x^\dag) h\|\le M \|h\|_{-a}, \quad h\in X
$$
From (a) and (\ref{1000}) it follows for $s\ge -a$ that
$\|F'(x) L^{-s}\|_{X\to Y}\le M \theta^{a+s}$ for all $x\in B_\rho(x^\dag)$. Thus
$\|F'(x)L^{-s}\|_{X\to Y}$ is uniformly bounded over $B_\rho(x^\dag)$. By
multiplying (\ref{1}) by a sufficiently small number, we may assume that
$F$ is properly scaled so that condition (b) is satisfied. Furthermore,
condition (a) implies that $F'(x)^*$ maps $Y$ into $X_b$ for $b\le a$ and
$\|F'(x)^*\|_{Y\to X_b}\le M \theta^{a-b}$ for all $x\in B_\rho(x^\dag)$.
Condition (c) says that $F'(x)^*$ is locally H\"{o}lder
continuous around $x^\dag$ with exponent $0<\beta\le 1$
when considered as operators from $Y$ to $X_b$. It is equivalent to
$$
\|L^b[F'(x)^*-F'(x^\dag)^*]\|_{Y\to X}\le K_0\|x-x^\dag\|^\beta, \quad x\in B_\rho(x^\dag)
$$
or
$$
\|[F'(x)-F'(x^\dag)] L^b\|_{X\to Y}\le K_0\|x-x^\dag\|^\beta, \quad x\in B_\rho(x^\dag).
$$
Condition (c) was used first in \cite{N2000} for the convergence
analysis of Landweber iteration in Hilbert scales.
It is easy to see that when $b=0$ and $\beta=1$,
this is exactly the Lipschitz condition on $F'(x)$. When $F$ satisfies (\ref{A10}),
(c) holds with $b=a$ and $\beta=1$. In \cite{N2000} it has been shown that
(c) implies
\begin{equation}\label{5.9.3}
\|F(x)-y-F'(x^\dag)(x-x^\dag)\|\le K_0\|x-x^\dag\|^\beta \|x-x^\dag\|_{-b}
\end{equation}
which follows easily from the identity
$$
F(x)-y-F'(x^\dag)(x-x^\dag)=\int_0^1 \left[F'(x^\dag+t(x-x^\dag))-F'(x^\dag)\right] L^b L^{-b} (x-x^\dag) dt.
$$

In this paper we will derive, under the above assumptions on $\{\a_n\}$,
$\{g_\a\}$ and $F$, the rate of convergence of $x_{n_\d}$ to $x^\dag$ as
$\delta\rightarrow 0$ when $e_0:=x_0-x^\dag$ satisfies the smoothness condition
\begin{equation}\label{33}
x_0-x^\dag \in X_\mu \quad \mbox{with } \frac{a-b}{\beta}<\mu\le b+2s,
\end{equation}
where $n_\d$ is the integer determined by the discrepancy
principle (\ref{DP}) with $\tau>1$.

The following consequence of the above assumptions on $F$ and $\{g_\a\}$ plays a crucial
role in the convergence analysis.

\begin{lemma}\label{L10}
Let $\{g_\a\}$ satisfy Assumptions \ref{A4} and \ref{A5}, let
$F$ satisfy Assumption \ref{A1}, and let $\{\a_n\}$ be a sequence of positive numbers.
Let $A=F'(x^\dag) L^{-s}$ and for any $x\in B_\rho(x^\dag)$ let $A_x=F'(x) L^{-s}$.
Then for $-\frac{b+s}{2(a+s)} \le \nu\le 1/2$ there holds
\begin{footnote}
{Throughout this paper we will always use $C$ to denote a generic
constant independent of $\d$ and $n$. We will also use the
convention $\Phi\lesssim \Psi$ to mean that $\Phi\le C \Psi$ for
some generic constant $C$.}
\end{footnote}
\begin{align*}
&\left\| (A^*A)^\nu \prod_{k=j+1}^n r_{\a_k}(A^*A) \left[ g_{\a_j}(A^*A)A^*
-g_{\a_j}(A_x^*A_x)A_x^*\right] \right\| \nonumber\\
&\qquad\qquad \qquad\qquad\qquad\qquad
\lesssim \frac{1}{\a_j} (s_n-s_{j-1})^{-\nu-\frac{b+s}{2(a+s)}} K_0\|x-x^\dag\|^\beta
\end{align*}
for $j=0, 1, \cdots, n$.
\end{lemma}

\begin{proof} Let $\eta_\a(\lambda)=(\a+\lambda)^{-1}$ and $\varphi_\a(\lambda)=g_\a(\lambda)-(\a+\lambda)^{-1}$.
We can write
$$
(A^*A)^\nu \prod_{k=j+1}^n r_{\a_k}(A^*A) \left[ g_{\a_j}(A^*A)A^*
-g_{\a_j}(A_x^*A_x)A_x^*\right] =J_1+J_2 +J_3,
$$
where
\begin{align*}
J_1&:=(A^*A)^\nu \prod_{k=j+1}^n r_{\a_k}(A^*A) g_{\a_j}(A^*A) [A^*-A_x^*],\\
J_2&:=(A^*A)^\nu \prod_{k=j+1}^n  r_{\a_k}(A^*A) \left[ \eta_{\a_j}(A^*A)-\eta_{\a_j}(A_x^* A_x)\right] A_x^*,\\
J_3&:=(A^*A)^\nu \prod_{k=j+1}^n  r_{\a_k}(A^*A) \left[ \varphi_{\a_j}(A^*A)-\varphi_{\a_j}(A_x^* A_x)\right] A_x^*.
\end{align*}
It suffices to show that the desired estimates hold for the norms of $J_1$, $J_2$ and $J_3$.

From (\ref{2.3}), (\ref{g2}) in Assumption \ref{A5} and Assumption \ref{A1}
it follows that
\begin{align*}
\|J_1\| & \lesssim
\left\|(A^*A)^\nu \prod_{k=j+1}^n r_{\a_k}(A^*A) g_{\a_j}(A^*A) (A^*A)^{\frac{b+s}{2(a+s)}}\right\|\\
&\qquad\qquad\qquad\qquad\qquad \qquad\qquad \times \left\|(A^*A)^{-\frac{b+s}{2(a+s)}}[A_x^*-A^*]\right\|\\
& \lesssim\sup_{0\le \lambda\le 1} \left(\lambda^{\nu+\frac{b+s}{2(a+s)}}
g_{\a_j}(\lambda) \prod_{k=j+1}^n r_{\a_k}(\lambda)\right) \|L^b [F'(x)^*-F'(x^\dag)^*]\|_{Y\to X}\\
& \lesssim \frac{1}{\a_j} (s_n-s_{j-1})^{-\nu-\frac{b+s}{2(a+s)}} K_0\|x-x^\dag\|^\beta
\end{align*}
which is the desired estimate.

In order to estimate $\|J_2\|$, we note that
\begin{align*}
\eta_{\a_j}(A^*A)-\eta_{\a_j}(A_x^*A_x)
&=(\a_j I+A^*A)^{-1} A^*(A_x-A) (\a_j I+A_x^*A_x)^{-1}\\
&\quad \, +(\a_j I+A^*A)^{-1} (A_x^*-A^*) A_x(\a_j I+A_x^*A_x)^{-1}.
\end{align*}
Therefore $J_2=J_2^{(1)}+J_2^{(2)}$, where
\begin{align*}
J_2^{(1)} &= (A^*A)^\nu \prod_{k=j+1}^n r_{\a_k}(A^*A) (\a_j I+A^*A)^{-1} A^*(A_x-A) (\a_j I+A_x^*A_x)^{-1} A_x^*, \\
J_2^{(2)} &= (A^*A)^\nu \prod_{k=j+1}^n r_{\a_k}(A^*A) (\a_j I+A^*A)^{-1} (A_x^*-A^*) A_x A_x^* (\a_j I+A_x A_x^*)^{-1}.
\end{align*}
With the help of Assumption \ref{A1} and (\ref{2.3}) we have
for any $w\in Y$ that
\begin{align*}
\|(A_x-A)(\a_j I&+A_x^*A_x)^{-1} A_x^* w\|\\
&=\|[F'(x)-F'(x^\dag)] L^b L^{-(b+s)}(\a_j I +A_x^*A_x)^{-1}A_x^* w\|\\
&\le K_0\|x-x^\dag\|^\beta \|(\a_j I+A_x^*A_x)^{-1} A_x^* w\|_{-(b+s)}\\
&\lesssim K_0\|x-x^\dag\|^\beta \|(A_x^*A_x)^{\frac{b+s}{2(a+s)}} (\a_j I + A_x^* A_x)^{-1} A_x^* w\|\\
&\lesssim K_0\|x-x^\dag\|^\beta \a_j^{-\frac{1}{2}+\frac{b+s}{2(a+s)}}\|w\|.
\end{align*}
This implies
\begin{equation}\label{7.10.1}
\|(A_x-A)(\a_j I +A_x^*A_x)^{-1} A_x^*\|
\lesssim K_0\|x-x^\dag\|^\beta \a_j^{-\frac{1}{2}+\frac{b+s}{2(a+s)}}.
\end{equation}
Thus, by using Lemma \ref{L20}, we derive
\begin{align*}
\|J_2^{(1)}\| &\le \sup_{0\le\lambda\le 1}
\left(\lambda^{\nu+\frac{1}{2}} (\a_j+\lambda)^{-1} \prod_{k=j+1}^n r_{\a_k}(\lambda)\right)
\|(A_x-A)(\a_j I +A_x^*A_x)^{-1} A_x^*\|\\
&\lesssim \a_j^{\nu-1+\frac{b+s}{2(a+s)}} \left(1+\a_j(s_n-s_j)\right)^{-\nu-\frac{1}{2}}
 K_0\|x-x^\dag\|^\beta.
\end{align*}
By using Assumption \ref{A1}, Lemma \ref{L20} and a similar argument in estimating $J_1$ we can derive
\begin{align*}
\|J_2^{(2)}\| & \lesssim\sup_{0\le \lambda\le 1} \left(\lambda^{\nu+\frac{b+s}{2(a+s)}}
(\a_j+\lambda)^{-1} \prod_{k=j+1}^n r_{\a_k}(\lambda)\right) \|L^b [F'(x)^*-F'(x^\dag)^*]\|_{Y\to X}\\
& \lesssim \a_j^{\nu-1+\frac{b+s}{2(a+s)}} \left(1+\a_j(s_n-s_j)\right)^{-\nu-\frac{b+s}{2(a+s)}} K_0\|x-x^\dag\|^\beta.
\end{align*}
Combining the above estimates on $J_2^{(1)}$ and $J_2^{(2)}$
and noting $\frac{b+s}{2(a+s)}\le \frac{1}{2}$, it follows that
\begin{align*}
\|J_2\| &\lesssim \a_j^{\nu-1+\frac{b+s}{2(a+s)}} \left(1+\a_j(s_n-s_j)\right)^{-\nu-\frac{b+s}{2(a+s)}}
K_0\|x-x^\dag\|^\beta\\
&=\frac{1}{\a_j} \left(s_n-s_{j-1}\right)^{-\nu-\frac{b+s}{2(a+s)}} K_0\|x-x^\dag\|^\beta.
\end{align*}

It remains to estimate $J_3$.  Since Assumption \ref{A4} implies
that $\varphi_{\a_j}(z)$ is analytic in $D_{\a_j}$, we have from
the Riesz-Dunford formula that
\begin{align}\label{300}
J_3 =\frac{1}{2\pi i} \int_{\Gamma_{\a_j}} \varphi_{\a_j}(z) T_j(z) dz,
\end{align}
where
$$
T_j(z):=(A^*A)^\nu \prod_{k=j+1}^n r_{\a_k}(A^*A)
\left[(z I-A^*A)^{-1}-(z I -A_x^* A_x)^{-1}\right] A_x^*.
$$
We can write $T_j(z)=T_j^{(1)}(z)+T_j^{(2)}(z)$, where
\begin{align*}
T_j^{(1)}(z)&:=(A^*A)^\nu \prod_{k=j+1}^n r_{\a_k}(A^*A)
(z I-A^*A)^{-1}A^*(A-A_x) (z I -A_x^* A_x)^{-1} A_x^*,\\
T_j^{(2)}(z)&:=(A^*A)^\nu \prod_{k=j+1}^n r_{\a_k}(A^*A)
(z I-A^*A)^{-1}(A^*-A_x^*) A_x A_x^*(z I -A_x A_x^*)^{-1}.
\end{align*}
We will estimate the norms of $T_j^{(1)}(z)$ and $T_j^{(2)}(z)$ for $z\in \Gamma_{\a_j}$.
With the help of Assumption \ref{A1}, (\ref{2.3}) and (\ref{200}), similar to the derivation
of (\ref{7.10.1}) we have
\begin{align*}
\|(A-A_x)(z I&-A_x^*A_x)^{-1} A_x^* \|
\lesssim K_0\|x-x^\dag\|^\beta |z|^{-\frac{1}{2}+\frac{b+s}{2(a+s)}}.
\end{align*}
Since $|z|\ge \a_j/2$ and $|z-\la|^{-1}\le b_0 (|z|+\la)^{-1}$ for $z\in \Gamma_{\a_j}$,
we have from (\ref{g3}) in Lemma \ref{L20} that
\begin{align*}
\|T_j^{(1)}(z)\| & \lesssim K_0\|x-x^\dag\|^\beta |z|^{-\frac{1}{2}+\frac{b+s}{2(a+s)}}\sup_{0\le \la \le 1}
\left(\la^{\nu+\frac{1}{2}} |z-\la|^{-1} \prod_{k=j+1}^n r_{\a_k}(\la)\right)\\
&\lesssim K_0\|x-x^\dag\|^\beta |z|^{-\frac{1}{2}+\frac{b+s}{2(a+s)}}\sup_{0\le \la \le 1}
\left(\la^{\nu+\frac{1}{2}} (|z|+\lambda)^{-1} \prod_{k=j+1}^n r_{\a_k}(\la)\right)\\
& \lesssim K_0\|x-x^\dag\|^\beta |z|^{\nu-1+\frac{b+s}{2(a+s)}} \left(1+(s_n-s_j)|z|\right)^{-\nu-1/2} \\
& \lesssim K_0\|x-x^\dag\|^\beta \a_j^{\nu-1+\frac{b+s}{2(a+s)}} \left(1+(s_n-s_j)\a_j \right)^{-\nu-1/2}.
\end{align*}
Next, by using (\ref{g3}) in Lemma \ref{L20}, (\ref{2.3}),
Assumption \ref{A1}(a) and (\ref{200}), we have for $z\in
\Gamma_{\a_j}$ that
\begin{align*}
\|T_j^{(2)}(z)\| &\le \left\|(A^*A)^\nu \prod_{k=j+1}^n r_{\a_k}(A^*A) (z I-A^* A)^{-1} (A^*A)^{\frac{b+s}{2(a+s)}} \right\|\\
&\qquad \qquad\qquad\qquad \times \left\|(A^*A)^{-\frac{b+s}{2(a+s)}} (A^*-A_x^*) A_x A_x^* (z I-A_x A_x^*)^{-1} \right\|\\
&\lesssim \sup_{0\le \lambda\le 1} \left(\lambda^{\nu+\frac{b+s}{2(a+s)}} |z-\lambda|^{-1} \prod_{k=j+1}^n r_{\a_k}(\lambda)\right)
\|L^b (F'(x^\dag)^*-F'(x)^*)\|\\
&\lesssim K_0\|x-x^\dag\|^\beta \sup_{0\le \lambda\le 1}
\left(\lambda^{\nu+\frac{b+s}{2(a+s)}} (|z|+\lambda)^{-1} \prod_{k=j+1}^n r_{\a_k}(\lambda)\right)\\
&\lesssim K_0\|x-x^\dag\|^\beta |z|^{\nu-1+\frac{b+s}{2(a+s)}} (1+(s_n-s_j)|z|)^{-\nu-\frac{b+s}{2(a+s)}}\\
&\lesssim K_0\|x-x^\dag\|^\beta \a_j^{\nu-1+\frac{b+s}{2(a+s)}} (1+(s_n-s_j)\a_j)^{-\nu-\frac{b+s}{2(a+s)}}.
\end{align*}
Combining the above estimates on $T_j^{(1)}(z)$ and $T_j^{(2)}(z)$
and noting $\frac{b+s}{2(a+s)}\le \frac{1}{2}$, it follows for $z\in
\Gamma_{\a_j}$ that
\begin{align*}
\|T_j(z)\|&\lesssim K_0\|x-x^\dag\|^\beta \a_j^{\nu-1+\frac{b+s}{2(a+s)}} (1+(s_n-s_j)\a_j)^{-\nu-\frac{b+s}{2(a+s)}}\\
&=\frac{1}{\a_j} (s_n-s_{j-1})^{-\nu-\frac{b+s}{2(a+s)}} K_0\|x-x^\dag\|^\beta
\end{align*}
Therefore, it follows from (\ref{300}) and Assumption \ref{A4} that
\begin{align*}
\|J_3\| & \lesssim \frac{1}{\a_j} (s_n-s_{j-1})^{-\nu-\frac{b+s}{2(a+s)}}
K_0\|x-x^\dag\|^\beta \int_{\Gamma_{\a_j}} |\varphi_{\a_j}(z)||d z|\\
& \lesssim \frac{1}{\a_j} (s_n-s_{j-1})^{-\nu-\frac{b+s}{2(a+s)}}
K_0\|x-x^\dag\|^\beta.
\end{align*}
The proof is therefore complete. \hfill $\Box$
\end{proof}

\section{\bf Convergence analysis}
\setcounter{equation}{0}

We begin with the following lemma.

\begin{lemma}\label{L2}
Let $\{\a_n\}$ be a sequence of positive numbers satisfying $\a_n\le c_1$, and let
$s_n$ be defined by (\ref{a1}).
Let $p\ge 0$ and $q\ge0$ be two numbers. Then we have
$$
\sum_{j=0}^n \frac{1}{\a_j} (s_n-s_{j-1})^{-p} s_j^{-q}\le C_0
s_n^{1-p-q} \left\{\begin{array}{lll}
1, & \max\{p,q\}<1,\\
\log (1+s_n), & \max\{p, q\}=1,\\
s_n^{\max\{p, q\}-1}, & \max\{p, q\}>1,
\end{array}\right.
$$
where $C_0$ is a constant depending only on $c_1$, $p$ and $q$.
\end{lemma}

\begin{proof} This result is essentially contained in
\cite[Lemma 4.3]{H2010} and its proof. For completeness,
we include here the proof with a simplified argument.
We first rewrite
$$
\sum_{j=0}^n \frac{1}{\a_j} (s_n-s_{j-1})^{-p} s_j^{-q}
=s_n^{1-p-q} \sum_{j=0}^n \frac{1}{\a_j s_n}\left(1-\frac{s_{j-1}}{s_n}\right)^{-p}
\left(\frac{s_j}{s_n}\right)^{-q}.
$$
Observe that when $0\le s_{j-1}/s_n\le 1/2$ we have
$$
\left(1-\frac{s_{j-1}}{s_n}\right)^{-p} \left(\frac{s_j}{s_n}\right)^{-q}
\le 2^p \left(\frac{s_j}{s_n}\right)^{-q}
$$
while when $s_{j-1}/s_n\ge 1/2$ we have
$$
\left(1-\frac{s_{j-1}}{s_n}\right)^{-p} \left(\frac{s_j}{s_n}\right)^{-q}
\le 2^q  \left(1-\frac{s_{j-1}}{s_n}\right)^{-p}.
$$
Consequently there holds with $C_{p,q}=\max\{2^p, 2^q\}$
\begin{align}\label{7.12.1}
\sum_{j=0}^n \frac{1}{\a_j} &(s_n-s_{j-1})^{-p} s_j^{-q} \nonumber\\
&\le C_{p,q} s_n^{1-p-q} \left(\sum_{j=0}^n \frac{1}{\a_j s_n} \left(\frac{s_j}{s_n}\right)^{-q}
+\sum_{j=0}^n \frac{1}{\a_j s_n} \left(1-\frac{s_{j-1}}{s_n}\right)^{-p}\right).
\end{align}
Note that $s_j-s_{j-1}=1/\a_j$, we have with $h=\frac{1}{2\a_0 s_n}$
\begin{align*}
\int_{s_0/s_n-h}^1 t^{-q} dt
&=\sum_{j=1}^n \int_{s_{j-1}/s_n}^{s_j/s_n} t^{-q} dt +\int_{s_0/s_n-h}^{s_0/s_n} t^{-q} dt\\
&\ge \sum_{j=1}^n \left(\frac{s_j}{s_n}\right)^{-q} \frac{s_j-s_{j-1}}{s_n}
+\frac{1}{2\a_0 s_n} \left(\frac{s_0}{s_n}\right)^{-q} \\
&\ge \frac{1}{2} \sum_{j=0}^n \frac{1}{\a_j s_n} \left(\frac{s_j}{s_n}\right)^{-q}.
\end{align*}
Therefore
\begin{align}\label{7.12.2}
\sum_{j=0}^n \frac{1}{\a_j s_n} \left(\frac{s_j}{s_n}\right)^{-q}
&\le 2 \int_{s_0/s_n-h}^1 t^{-q} dt
\le \left\{\begin{array}{lll}
\frac{2}{1-q}, & q<1,\\
2\log (2\a_0 s_n), & q=1,\\
\frac{2}{q-1} (2\a_0 s_n)^{q-1}, \quad & q>1.
\end{array}\right.
\end{align}
By a similar argument we have with $h=\frac{1}{2\a_n s_n}$
\begin{align}\label{7.12.3}
\sum_{j=0}^n \frac{1}{\a_j s_n} \left(1-\frac{s_{j-1}}{s_n}\right)^{-p}
&\le 2 \int_0^{\frac{s_{n-1}}{s_n}+h}(1- t)^{-p} dt
\le \left\{\begin{array}{lll}
\frac{2}{1-p}, & p<1,\\
2\log (2\a_n s_n), & p=1,\\
\frac{2}{p-1} (2\a_n s_n)^{p-1},  & p>1.
\end{array}\right.
\end{align}
Combining (\ref{7.12.1}), (\ref{7.12.2}) and (\ref{7.12.3}) and using the condition $\a_n\le c_1$,
we obtain the desired inequalities. \hfill $\Box$
\end{proof}

In order to derive the necessary estimates on $x_n-x^\dag$, we need
some useful identities. For simplicity of presentation, we set
$$
e_n:=x_n-x^\dag, \quad  A:=F'(x^\dag) L^{-s} \quad \mbox{and} \quad A_n:=F'(x_n) L^{-s}.
$$
It follows from (\ref{2}) and (\ref{2.2}) that
\begin{align*}
e_{n+1}&=e_n- L^{-s} g_{\a_n}\left(A_n^* A_n\right) A_n^* (F(x_n)-y^\d).
\end{align*}
Let
$$
u_n:=F(x_n)-y-F'(x^\dag) (x_n-x^\dag).
$$
Then we can write
\begin{align}\label{400}
e_{n+1}&= e_n-L^{-s} g_{\a_n}(A^*A) A^* (F(x_n)-y^\d) \nonumber\\
&\quad\, -L^{-s} \left[g_{\a_n}(A_n^*A_n)A_n^* -g_{\a_n}(A^*A) A^*\right] (F(x_n)-y^\d) \nonumber\\
&=L^{-s} r_{\a_n}(A^*A) L^s e_n -L^{-s} g_{\a_n}(A^*A) A^* (y-y^\d+ u_n) \nonumber\\
&\quad\, -L^{-s} \left[g_{\a_n}(A_n^*A_n)A_n^* -g_{\a_n}(A^*A)A^* \right] (F(x_n)-y^\d).
\end{align}
By telescoping (\ref{400}) we can obtain
\begin{align}\label{20}
e_{n+1} &=L^{-s} \prod_{j=0}^n r_{\a_j}(A^* A) L^{s} e_0 \nonumber\\
&\quad\,  -L^{-s} \sum_{j=0}^n \prod_{k=j+1}^n r_{\a_k}(A^*A) g_{\a_j}(A^*A) A^*(y-y^\d+u_j) \nonumber\\
&\quad\, -L^{-s} \sum_{j=0}^n \prod_{k=j+1}^n r_{\a_k}(A^*A)
\left[g_{\a_j}(A_j^*A_j)A_j^*-g_{\a_j}(A^*A)A^*\right] (F(x_j)-y^\d).
\end{align}
By multiplying  (\ref{20}) by $T:=F'(x^\dag)$ and noting that $A=T L^{-s}$ and
$$
I-\sum_{j=0}^n \prod_{k=j+1}^n r_{\a_k}(AA^*)
g_{\a_j}(AA^*) AA^*=\prod_{j=0}^n r_{\a_j}(AA^*),
$$
we can obtain
\begin{align}\label{21}
T &e_{n+1}-y^\d+y \nonumber\\
& =A \prod_{j=0}^n r_{\a_j}(A^*A) L^s e_0
 +\prod_{j=0}^n r_{\a_j}(AA^*)(y-y^\d) \nonumber\\
&\quad\,  - \sum_{j=0}^n \prod_{k=j+1}^n r_{\a_k}(AA^*) g_{\a_j}(AA^*) AA^*u_j \nonumber\\
&\quad\, - \sum_{j=0}^n A \prod_{k=j+1}^n
r_{\a_k}(A^*A) \left[g_{\a_j}(A_j^*A_j)A_j^*-g_{\a_j}(A^*A)A^*\right] (F(x_j)-y^\d).
\end{align}

Based on (\ref{20}) and (\ref{21}) we will derive the order optimal
convergence rate of $x_{n_\d}$ to $x^\dag$ when $e_0:=x_0-x^\dag$
satisfies the smoothness condition (\ref{33}). Under such condition
we have $L^s e_0\in X_{\mu-s}$ and $|\frac{\mu-s}{a+s}|\le 1$. Thus,
with the help of Assumption \ref{A1}(a), it follows from (\ref{2.5})
and (\ref{2.3}) that there exists $\omega\in X$ such that
\begin{equation}\label{31}
L^s e_0=(A^* A)^{\frac{\mu-s}{2(a+s)}} \omega \quad \mbox{and}\quad
c_2 \|\omega\|\le \|e_0\|_\mu\le c_3 \|\omega\|
\end{equation}
for some generic constants $c_3\ge c_2>0$. We will first derive the
crucial estimates on $\|e_n\|_\mu$ and $\|Te_n\|$. To this end, we
introduce the integer $\tilde{n}_\d$ satisfying
\begin{equation}
s_{\tilde{n}_\d}^{-\frac{a+\mu}{2(a+s)}}\le \frac{(\tau-1)\d}{ 2 c_0 \|\omega\|}< s_n^{-\frac{a+\mu}{2(a+s)}},
\quad 0\le n<\tilde{n}_\d,
\end{equation}
where $c_0>1$ is the constant appearing in (\ref{60}).
Such $\tilde{n}_\d$ is well-defined since $s_n\rightarrow \infty$
as $n\rightarrow \infty$.

\begin{proposition}\label{P1}
Let $F$ satisfy Assumptions \ref{A1}, let $\{g_\a\}$ satisfy
Assumptions \ref{A4} and \ref{A5}, and let $\{\a_n\}$ be a sequence
of positive numbers satisfying (\ref{60}). If $e_0\in X_\mu$ for
some $(a-b)/\beta< \mu\le b+2s$ and if $K_0 \|\omega\|^\beta$ is
suitably small, then there exists a generic constant $C_*>0$ such
that
\begin{align}\label{30}
\|e_n\|_\mu \le C_* \|\omega\| \qquad \mbox{and} \qquad
\|T e_n\| \le C_* s_n^{-\frac{a+\mu}{2(a+s)}} \|\omega\|
\end{align}
and
\begin{equation}\label{50}
\|T e_n-y^\d+y\|\le (c_0 +C_* K_0\|\omega\|^\beta)
s_n^{-\frac{a+\mu}{2(a+s)}} \|\omega\| +\d
\end{equation}
for all $0\le n\le \tilde{n}_\d$.
\end{proposition}

\begin{proof}
We will show (\ref{30}) by induction. By using (\ref{31}) and
$\|A\|\le \sqrt{\a_0}$ we have
$$
\|T e_0\|=\|A L^s e_0\|=\|(A^*A)^{1/2} L^s e_0\|=\|(A^*A)^{\frac{a+\mu}{2(a+s)}} \omega\|
\le \a_0^{\frac{a+\mu}{2(a+s)}} \|\omega\|.
$$
This together with (\ref{31}) shows (\ref{30}) for $n=0$ if $C_*\ge \max\{1, c_3\}$.
Next we assume that (\ref{30}) holds
for all $0\le n\le l$ for some $l<\tilde{n}_\d$ and we are going to show (\ref{30}) holds for $n=l+1$.

With the help of (\ref{2.3}) and (\ref{31}) we can derive from (\ref{20}) that
\begin{align*}
&\|e_{l+1}\|_\mu \\
&\lesssim  \left\|\prod_{j=0}^l r_{\a_j}(A^*A) \omega\right\| +
\left\|\sum_{j=0}^l  (AA^*)^{\frac{a+2s-\mu}{2(a+s)}}
g_{\a_j}(AA^*) \prod_{k=j+1}^l r_{\a_k} (AA^*) (y-y^\d+u_j)\right\| \\
& +\left\| \sum_{j=0}^l (A^*A)^{\frac{s-\mu}{2(a+s)}}
\prod_{k=j+1}^l r_{\a_k}(A^*A)
\left[g_{\a_j}(A_j^*A_j)A_j^*-g_{\a_j}(A^*A)A^*\right]
(F(x_j)-y^\d)\right\|.
\end{align*}
Since $(a-b)/\beta<\mu\le b+2s$ and $0\le b\le a$, we have
$$
0\le \frac{a+2s-\mu}{2(a+s)}<1 \quad \mbox{and} \quad
-\frac{b+s}{2(a+s)}\le \frac{s-\mu}{2(a+s)}< \frac{1}{2}.
$$
Thus we may use Assumption \ref{A5} and Lemma \ref{L10} to conclude
\begin{align}\label{40}
\|e_{l+1}\|_\mu &\lesssim  \|\omega\| + \sum_{j=0}^l \frac{1}{\a_j}
(s_l-s_{j-1})^{-\frac{a+2s-\mu}{2(a+s)}}(\d+\|u_j\|) \nonumber\\
&\quad\, + \sum_{j=0}^l \frac{1}{\a_j}
(s_l-s_{j-1})^{-\frac{b+2s-\mu}{2(a+s)}} K_0\|e_j\|^\beta
\|F(x_j)-y^\d\|.
\end{align}
Moreover, by using (\ref{31}), Assumption \ref{A5} and Lemma
\ref{L10}, we have from (\ref{21}) that
\begin{align}\label{41}
\|T e_{l+1}-y^\d+y\| &\le s_l^{-\frac{a+\mu}{2(a+s)}} \|\omega\|+\d
+ b_2 \sum_{j=0}^l \frac{1}{\a_j} (s_l-s_{j-1})^{-1} \|u_j\| \nonumber\\
&\quad\, + c_4 \sum_{j=0}^l \frac{1}{\a_j}
(s_l-s_{j-1})^{-\frac{b+a+2s}{2(a+s)}} K_0\|e_j\|^\beta
\|F(x_j)-y^\d\|,
\end{align}
where $c_4>0$ is a generic constant.

By using the interpolation inequality (\ref{2.1}), Assumption
\ref{A1}(a) and the induction hypotheses, it follows for all $0\le
j\le l$ that
\begin{equation}\label{e10}
\|e_j\|\le \|e_j\|_{-a}^{\frac{\mu}{a+\mu}}
\|e_j\|_\mu^{\frac{a}{a+\mu}} \lesssim \|T e_j\|^{\frac{\mu}{a+\mu}}
\|e_j\|_\mu^{\frac{a}{a+\mu}}\lesssim \|\omega\|
s_j^{-\frac{\mu}{2(a+s)}}.
\end{equation}
With the help of (\ref{5.9.3}) and the interpolation inequality
(\ref{2.1}), we have
\begin{equation}\label{e11}
\|u_j\|\le K_0\|e_j\|^\beta \|e_j\|_{-b}\le
K_0\|e_j\|_{-a}^{\frac{b+\mu+\mu\beta}{a+\mu}}
\|e_j\|_\mu^{\frac{a+a\beta-b}{a+\mu}}.
\end{equation}
We then obtain from Assumption \ref{A1}(a) and the induction hypotheses that
\begin{equation}\label{e12}
\|u_j\|\lesssim K_0\|T e_j\|^{\frac{b+\mu+\mu \beta}{a+\mu}} \|e_j\|_\mu^{\frac{a+a\beta-b}{a+\mu}}
\lesssim K_0 \|\omega\|^{1+\beta} s_j^{-\frac{b+\mu+\mu \beta}{2(a+s)}}.
\end{equation}
On the other hand, since (\ref{1000}) and the induction hypotheses implies
$$
\|e_j\|_{-a}\lesssim \|e_j\|_\mu\lesssim \|\omega\|, \qquad 0\le j\le l
$$
and since $\mu >(a-b)/\beta$, we have from (\ref{e11}) and
Assumption \ref{A1}(a) that
\begin{align}\label{e13}
\|u_j\| \lesssim K_0\|e_j\|_{-a}
\|e_j\|_{-a}^{\frac{b-a+\mu\beta}{a+\mu}}
\|e_j\|_\mu^{\frac{a+a\beta-b}{a+\mu}} \lesssim K_0\|\omega\|^\beta
\|T e_j\|.
\end{align}
Therefore, by using the fact
\begin{equation}\label{77}
\d\le \frac{2 c_0}{\tau-1} \|\omega\| s_j^{-\frac{a+\mu}{2(a+s)}},
\qquad 0\le j\le l
\end{equation}
and the induction hypotheses we have
\begin{equation}\label{e14}
\|F(x_j)-y^\d\|\le \d +\|T e_j\| +\|u_j\| \lesssim \|\omega\|
s_j^{-\frac{a+\mu}{2(a+s)}}.
\end{equation}
In view of the estimates (\ref{e10}), (\ref{e12}), (\ref{e14}) and the inequality
$$
\sum_{j=0}^l \frac{1}{\a_j} (s_l-s_{j-1})^{-\frac{a+2s-\mu}{2(a+s)}}\lesssim
s_l^{\frac{a+\mu}{2(a+s)}}
$$
which follows from Lemma \ref{L2},
we have from (\ref{40}) and (\ref{41}) that
\begin{align*}
\|e_{l+1}\|_\mu & \le c_5 \|\omega\| + c_5 s_l^{\frac{a+\mu}{2(a+s)}} \d \\
&\quad\, +CK_0\|\omega\|^{1+\beta} \sum_{j=0}^l \frac{1}{\a_j}
(s_l-s_{j-1})^{-\frac{a+2s-\mu}{2(a+s)}} s_j^{-\frac{b+\mu+\mu\beta}{2(a+s)}}\\
&\quad\, + CK_0\|\omega\|^{1+\beta} \sum_{j=0}^l \frac{1}{\a_j}
(s_l-s_{j-1})^{-\frac{b+2s-\mu}{2(a+s)}} s_j^{-\frac{a+\mu+\mu\beta}{2(a+s)}}
\end{align*}
and
\begin{align*}
\|T e_{l+1}-y^\d+y\| &\le \|\omega\| s_l^{-\frac{a+\mu}{2(a+s)}} +\d \nonumber\\
&\quad\, + CK_0\|\omega\|^{1+\beta} \sum_{j=0}^l \frac{1}{\a_j}
(s_l-s_{j-1})^{-1} s_j^{-\frac{b+\mu+\mu\beta}{2(a+s)}} \nonumber\\
&\quad\, +CK_0\|\omega\|^{1+\beta} \sum_{j=0}^l \frac{1}{\a_j} (s_l-s_{j-1})^{-\frac{b+a+2s}{2(a+s)}}
s_j^{-\frac{a+\mu+\mu\beta}{2(a+s)}},
\end{align*}
where $c_5$ and $C$ are two positive generic constants.

With the help of Lemma \ref{L2}, $\mu>(a-b)/\beta$, (\ref{77})
and (\ref{60}) we have
\begin{align*}
\|e_{l+1}\|_\mu \le \left(c_5  +\frac{2}{\tau-1} c_0 c_5 +C K_0\|\omega\|^\beta\right) \|\omega\|,
\end{align*}
and
\begin{align}\label{80}
\|Te_{l+1}-y^\d+y\| &\le \d+ \left(1 +CK_0\|\omega\|^\beta\right)
\|\omega\| s_l^{-\frac{a+\mu}{2(a+s)}} \nonumber\\
&\le \d+ c_0 \left(1+CK_0\|\omega\|^\beta\right) \|\omega\|
s_{l+1}^{-\frac{a+\mu}{2(a+s)}}.
\end{align}
Consequently $\|e_{l+1}\|_\mu \le C_* \|\omega\|$ if $C_*\ge
2c_5+\frac{2}{\tau-1} c_0 c_5$ and $K_0\|\omega\|^\beta$ is suitably
small. Moreover, from (\ref{80}), (\ref{77}) and (\ref{60}) we also
have
\begin{align*}
\|T e_{l+1}\| &\le 2\d + c_0\left(1+ CK_0\|\omega\|^\beta\right)
\|\omega\| s_{l+1}^{-\frac{a+\mu}{2(a+s)}}\\
&\le  \left(\frac{4 c_0^2}{\tau-1} +c_0 +CK_0\|\omega\|^\beta\right)
\|\omega\| s_{l+1}^{-\frac{a+\mu}{2(a+s)}}\\
&\le C_* \|\omega\| s_{l+1}^{-\frac{a+\mu}{2(a+s)}}
\end{align*}
if $C_*\ge 2 c_0 +\frac{4 c_0^2}{\tau-1}$ and $K_0\|\omega\|^\beta$
is suitably small. We therefore complete the proof of (\ref{30}). In
the meanwhile, (\ref{80}) gives the proof of (\ref{50}). \hfill
$\Box$
\end{proof}

From Proposition \ref{P1} and its proof it follows that $x_n\in B_\rho(x^\dag)$ for
$0\le n\le \tilde{n}_\d$ if $\|\omega\|$ is sufficiently small.
Furthermore, from (\ref{e12}) and (\ref{e13}) we have
\begin{equation}\label{66}
\|F(x_n)-y-T e_n\|\lesssim K_0\|\omega\|^{1+\beta} s_n^{-\frac{b+\mu+\mu\beta}{2(a+s)}}
\end{equation}
and
\begin{equation}\label{67}
\|F(x_n)-y-T e_n\|\lesssim K_0\|\omega\|^\beta \|Te_n\|
\end{equation}
for $0\le n\le \tilde{n}_\d$.

In the following we will show that $n_\d\le \tilde{n}_\d$ for the
integer $n_\d$ defined by (\ref{DP}) with $\tau>1$.
Consequently, the method given by (\ref{2}) and (\ref{DP}) is well-defined.

\begin{lemma}\label{L44}
Let all the conditions in Proposition \ref{P1} hold. Let $\tau>1$ be a given number.
If $e_0\in X_\mu$ for some $(a-b)/\beta<\mu\le b+2s$ and if $K_0\|e_0\|_\mu^\beta$ is suitably small, then the
discrepancy principle (\ref{DP}) defines a finite integer $n_\d$ satisfying $n_\d\le \tilde{n}_\d$.
\end{lemma}

\begin{proof}
From Proposition \ref{P1}, (\ref{66}) and $\mu> (a-b)/\beta$
it follows for $0\le n\le \tilde{n}_\d$ that
\begin{align*}
\|F(x_n)-y^\d\| & \le \|F(x_n)-y-T e_n\|+\|T e_n -y^\d+y\|\\
&\le C K_0\|\omega\|^{1+\beta} s_n^{-\frac{b+\mu+\mu\beta}{2(a+s)}}
+\left(c_0 +CK_0\|\omega\|^\beta\right) s_n^{-\frac{a+\mu}{2(a+s)}} \|\omega\|+\d\\
&\le \left(c_0 +CK_0\|\omega\|^\beta\right) s_n^{-\frac{a+\mu}{2(a+s)}} \|\omega\| +\d.
\end{align*}
By setting $n=\tilde{n}_\d$ in the above inequality and using the definition of $\tilde{n}_\d$ we obtain
\begin{align*}
\|F(x_{\tilde{n}_\d})-y^\d\|\le \left(1+\frac{\tau-1}{2}+ CK_0\|\omega\|^\beta\right) \d\le \tau \d
\end{align*}
if $K_0\|\omega\|^\beta$ is suitably small. According to the
definition of $n_\d$ we have $n_\d\le \tilde{n}_\d$. \hfill $\Box$
\end{proof}

Now we are ready to prove the main result concerning the order optimal convergence rates
for the method defined by (\ref{2}) and (\ref{DP}) with $\tau>1$.

\begin{theorem}\label{T1}
Let $F$ satisfy Assumptions \ref{A1}, let $\{g_\a\}$
satisfy Assumptions \ref{A4} and \ref{A5},  and let $\{\a_n\}$ be
a sequence of positive numbers satisfying (\ref{60}).
If $e_0\in X_\mu$ for some $(a-b)/\beta<\mu\le b+2s$ and if $K_0\|e_0\|_\mu^\beta$ is suitably small,
then for all $r\in [-a, \mu]$ there holds
$$
\|x_{n_\d}-x^\dag\|_r \le C \|e_0\|_\mu^{\frac{a+r}{a+\mu}} \d^{\frac{\mu-r}{a+\mu}}
$$
for the integer $n_\d$ determined by the discrepancy principle
(\ref{DP}) with $\tau>1$, where $C>0$ is a generic constant.
\end{theorem}

\begin{proof}
It follows from (\ref{67}) that if $K_0\|\omega\|^\beta$ is suitably small then
\begin{align*}
\|F(x_n)-y-T e_n\| \le \frac{1}{2} \|T e_n\|
\end{align*}
which implies $\|Te_n\|\le 2\|F(x_n)-y\|$ for $0\le n\le
\tilde{n}_\d$. Since Lemma \ref{L44} implies $n_\d\le \tilde{n}_\d$,
it follows from Assumption \ref{A1}(a) and the definition of $n_\d$
that
\begin{align*}
\|e_{n_\d}\|_{-a} \le \frac{1}{m}\|T e_{n_\d}\|\le
\frac{2}{m}\left(\|F(x_{n_\d})-y^\d\|+\d\right) \le
\frac{2(1+\tau)}{m}\d.
\end{align*}
But from Proposition \ref{P1} we have $\|e_{n_\d}\|_\mu\le C_*
\|\omega\|$. The desired estimate then follows from the
interpolation inequality (\ref{2.1}) and (\ref{31}). \hfill $\Box$
\end{proof}

\begin{remark}
If $F$ satisfies (\ref{A10}) and $\{x_n\}$ is defined by (\ref{2})
with $s>-a/2$, then the order optimal convergence rate holds for
$x_0-x^\dag\in X_\mu$ with $0<\mu\le a+2s$. On the other hand, if
$F'(x)$ satisfies the Lipschitz condition 
$$
\|F'(x)-F'(x^\dag)\|\le K_0\|x-x^\dag\|, \quad x\in B_\rho(x^\dag)
$$ 
and $\{x_n\}$ is defined by (\ref{2}) with $s>a/2$, then the order
optimal convergence rate holds for $x_0-x^\dag\in X_\mu$ with
$a<\mu\le 2s$.
\end{remark}

\section{\bf Examples}\label{Sect4}
\setcounter{equation}{0}

In this section we will give several important examples of $\{g_\a\}$ that satisfy Assumptions \ref{A4}
and \ref{A5}. Thus, Theorem \ref{T1} applies to the corresponding methods if $F$ satisfies
Assumption \ref{A1}  and $\{\a_n\}$ satisfies (\ref{60}).
For all these examples, the functions $g_\a$ are analytic at least in the domain
$$
D_\a:=\{z\in {\mathbb C}:  z\ne -\a, -1\}.
$$
Moreover, for each $\a>0$, we always take the closed contour
$\Gamma_\a$ to be (see \cite{BK04})
$$
\Gamma_\a=\Gamma_\a^{(1)}\cup \Gamma_\a^{(2)}\cup
\Gamma_\a^{(3)}\cup \Gamma_\a^{(4)},
$$
with
\begin{align*}
\Gamma_\a^{(1)}&:= \{ z=\frac{\a}{2} e^{i\phi}: \phi_0\le
\phi \le 2\pi-\phi_0\},\\
\Gamma_\a^{(2)}&:=\{z=R e^{i\phi}: -\phi_0\le \phi\le
\phi_0\},\\
\Gamma_\a^{(3)}&:=\{z=t e^{i\phi_0}: \a/2\le t\le R\},\\
\Gamma_\a^{(4)}&:=\{z=t e^{-i\phi_0}: \a/2\le t\le R\},
\end{align*}
where $R>\max\{1, \a\}$ and $0<\phi_0<\pi/2$ are fixed numbers.
Clearly $\Gamma_\a\subset D_\a$ and $[0, 1]$ lies inside
$\Gamma_\a$. It is straightforward to check that (\ref{2.6}) is satisfied.

\begin{example} We first consider for $\a>0$ the function
$g_\alpha$ given by
$$
g_\alpha(\lambda)
=\frac{(\a+\la)^N-\a^N}{\la (\a+\la)^N}
$$
where $N\ge 1$ is a fixed integer. This function arises from the
iterated Tikhonov regularization of order $N$ for linear ill-posed
problems. The corresponding method (\ref{2}) becomes
\begin{align*}
u_{n,0}&=x_n,\\
u_{n, l+1}&=u_{n,l}-\left(\a_n L^{2s} +T_n^* T_n\right)^{-1} T_n^*
\left(F(x_n)-y^\d -T_n(x_n-u_{n,l})\right), \\
&\qquad\qquad\qquad\qquad\qquad\qquad\qquad\qquad\qquad \quad l=0, \cdots, N-1,\\
x_{n+1}&= u_{n, N},
\end{align*}
where $T_n:=F'(x_n)$. When $N=1$, this is the Levenberg-Marquardt method in Hilbert scales.
The corresponding residual function is
$r_\alpha(\lambda)=\alpha^N (\alpha+\lambda)^{-N}$.
In order to verify Assumption \ref{A5}, we recall the inequality (see \cite[Lemma 3]{Jin10a})
$$
\la \prod_{k=j}^n \frac{\a_k}{\a_k+\la} \le (s_n-s_{j-1})^{-1} \quad \mbox{for all } \la\ge 0.
$$
Then for $0\le \nu\le 1$ and $\la\ge 0$ we have
$$
\la^\nu \prod_{k=j}^n r_{\a_k}(\la) \le \left(\la \prod_{k=j}^n \frac{\a_k}{\a_k+\la}\right)^\nu
\le (s_n-s_{j-1})^{-\nu}
$$
and
\begin{align*}
\la^\nu g_{\a_j}(\la) \prod_{k=j+1}^n r_{\a_k}(\la)
&=\frac{(\a_j+\la)^N-\a_j^N}{\a_j^N \la^{1-\nu}}
\prod_{k=j}^n \left(\frac{\a_k}{\a_k+\la}\right)^N\\
&=\sum_{l=0}^{N-1}\left(\begin{array}{ccc}
N\\
l
\end{array}\right) \a_j^{l-N} \la^{N+\nu-l-1}
\prod_{k=j}^n \left(\frac{\a_k}{\a_k+\la}\right)^N\\
&\le \sum_{l=0}^{N-1}\left(\begin{array}{ccc}
N\\
l
\end{array}\right) \a_j^{l-N} \left(\la
\prod_{k=j}^n \frac{\a_k}{\a_k+\la}\right)^{N+\nu-l-1}\\
&\le \sum_{l=0}^{N-1}\left(\begin{array}{ccc}
N\\
l
\end{array}\right) \a_j^{l-N} (s_n-s_{j-1})^{-N-\nu+l+1} \\
&\le C_N \frac{1}{\a_j} (s_n-s_{j-1})^{-\nu},
\end{align*}
where $C_N=2^N-1$ and we used the fact $\a_j^{-1}\le s_n-s_{j-1}$.
We therefore obtain (\ref{g1}) and (\ref{g2}) in Assumption \ref{A5}.

Next we will verify (\ref{2.8}) in Assumption \ref{A4}. Note that
$$
\varphi_\a(z)=\frac{\a(\a+z)^{N-1}-\a^N}{z (\a+z)^N}=\frac{1}{z (\a+ z)^N} \sum_{j=0}^{N-2}
\left(\begin{array}{ccc}
N-1\\
j
\end{array}\right) \a^{j+1} z^{N-1-j}.
$$
It is easy to check $|\varphi_\a(z)|\lesssim \a^{-1}$
on $\Gamma_\a^{(1)}$ and $|\varphi_\a(z)|\lesssim 1$ on $\Gamma_\a^{(2)}$.
Moreover, on $\Gamma_\a^{(3)}\cup \Gamma_\a^{(4)}$ there holds
\begin{align*}
|\varphi_\a(z)|\lesssim \frac{1}{t(\a+t)^N} \sum_{j=0}^{N-2} \a^{j+1} t^{N-1-j}
\lesssim \sum_{j=0}^{N-2} \a^{j+1} t^{-2-j}.
\end{align*}
Therefore
\begin{align*}
\int_{\Gamma_\a} |\varphi_\a(z)| |dz|
&= \int_{\Gamma_\a^{(1)}} |\varphi_\a(z)| |dz| +\int_{\Gamma_\a^{(2)}} |\varphi_\a(z)| |dz|
+\int_{\Gamma_\a^{(3)}\cup \Gamma_\a^{(4)}} |\varphi_\a(z)| |dz|\\
&\lesssim \a^{-1} \int_{\phi_0}^{2\pi-\phi_0} \a d \phi +\int_{-\phi_0}^{\phi_0} d\phi
+\sum_{j=0}^{N-2} \a^{j+1} \int_{\a/2}^R t^{-2-j} dt\\
&\lesssim 1.
\end{align*}
Assumption \ref{A4} is therefore verified.
\end{example}

\begin{example} We consider the method (\ref{2}) with $g_\a$ given by
$$
g_\alpha(\lambda)=\frac{1}{\lambda}\left(1-e^{-\lambda/\alpha}\right)
$$
which arises from the asymptotic regularization for linear
ill-posed problems. In this method, the iterative sequence $\{x_n\}$ is
equivalently defined as $x_{n+1}:=x(1/\a_n)$, where $x(t)$ is the unique solution of
the initial value problem
\begin{align*}
&\frac{d}{d t} x(t)=L^{-2s} F'(x_n)^* \left(y^\d-F(x_n)+F'(x_n)(x_n-x(t))\right), \quad t>0,\\
&x(0)=x_n.
\end{align*}
The corresponding residual function is
$r_\alpha(\lambda)=e^{-\lambda/\alpha}$. We first verify Assumption \ref{A5}. It is easy to see
$$
\la^\nu  \prod_{k=j}^n r_{\a_j}(\la) =\la^\nu e^{-\la (s_n-s_{j-1})}
\le \nu^\nu e^{-\nu} (s_n-s_{j-1})^{-\nu}\le (s_n-s_{j-1})^{-\nu}
$$
for $0\le \nu\le 1$ and $\la\ge 0$. This shows (\ref{g1}). By using the elementary inequality
$e^{-p \la} -e^{-q \la} \le (q-p)/q$ for $0<p\le q$ and $\la\ge 0$
and observing that $0\le r_\a(\la)\le 1$ and $0\le g_\a(\la)\le 1/\a$, we have for $0\le \nu\le 1$
and $\la\ge 0$ that
\begin{align*}
\la^\nu g_{\a_j}(\la) \prod_{k=j+1}^n r_{\a_k}(\la)
&\le \frac{1}{\a_j^{1-\nu}} \left(\la g_{\a_j}(\la) \prod_{k=j+1}^n r_{\a_k}(\la)\right)^\nu \\
&= \frac{1}{\a_j^{1-\nu}} \left( e^{-(s_n-s_j)\la} -e^{-(s_n-s_{j-1})\la}\right)^\nu\\
&\le \frac{1}{\a_j} (s_n-s_{j-1})^{-\nu}
\end{align*}
which gives (\ref{g2}).

In order to verify (\ref{2.8}) in Assumption \ref{A4}, we note that
$$
\varphi_\a(z)=\frac{1-e^{-z/\a}}{z}-\frac{1}{\a+z}=\frac{\a-(\a+z) e^{-z/\a}}{z(\a+z)}.
$$
It is easy to see that $|\varphi_\a(z)|\lesssim \a^{-1}$ on $\Gamma_\a^{(1)}$,
$|\varphi_\a(z)|\lesssim 1$ on $\Gamma_\a^{(2)}$ and
$$
|\varphi_\a(z)|\lesssim \frac{\a+(\a+t) e^{-\frac{t}{\a}\cos \phi_0}}{t(\a+t)}\lesssim \a t^{-2}
$$
on $\Gamma_\a^{(3)}\cup \Gamma_\a^{(4)}$. Therefore
$$
\int_{\Gamma_\a} |\varphi_\a(z)| |dz|\lesssim 1+ \int_{\a/2}^R \a t^{-2} dt \lesssim 1.
$$
\end{example}

\begin{example}\label{E3}
We consider for $0<\a\le 1$ the function $g_\alpha$ given by
$$
g_\alpha(\lambda)=\sum_{l=0}^{[1/\alpha]-1}(1-\lambda)^l=\frac{1-(1-\la)^{[1/\a]}}{\la}
$$
which arises from the linear Landweber iteration, where $[1/\a]$ denotes the largest integer not
greater than $1/\a$. The method (\ref{2}) then becomes
\begin{align*}
u_{n,0}&=x_n,\\
u_{n, l+1}&=u_{n,l}-L^{-2s} T_n^* \left(F(x_n)-y^\d -T_n(x_n-u_{n,l})\right), \quad 0\le l\le [1/\a_n]-1,\\
x_{n+1}&= u_{n, [1/\a_n]},
\end{align*}
where $T_n:=F'(x_n)$. When $\a_n=1$ for all $n$, this method reduces
to the Landweber iteration in Hilbert scales proposed in
\cite{N2000}. The corresponding residual function is
$r_\alpha(\lambda)=(1-\lambda)^{[1/\alpha]}$. We first verify
Assumption \ref{A5} when the sequence $\{\a_n\}$ is given by
$\a_n=1/k_n$ for some integers $k_n\ge 1$. Then for $0\le \nu\le 1$
and $0\le \la\le 1$ we have
$$
\la^\nu \prod_{k=j}^n r_{\a_k}(\la)= \la^\nu (1-\la)^{s_n-s_{j-1}}
\le \nu^\nu (s_n-s_{j-1})^{-\nu}\le (s_n-s_{j-1})^{-\nu}.
$$
We thus obtain (\ref{g1}). Observing that $0\le r_{\a_j}(\la)\le 1$ and $0\le g_{\a_j}(\la)\le 1/\a_j$
for $0\le \la \le 1$, we have
\begin{align*}
\la^\nu g_{\a_j}(\la) \prod_{k=j+1}^n r_{\a_k}(\la)
&\le \frac{1}{\a_j^{1-\nu}} \left(\la g_{\a_j}(\la) \prod_{k=j+1}^n r_{\a_k}(\la)\right)^\nu\\
&=\frac{1}{\a_j^{1-\nu}} \left((1-\la)^{s_n-s_j}-(1-\la)^{s_n-s_{j-1}}\right)^\nu.
\end{align*}
Thus, (\ref{g2}) follows from the elementary inequality
$t^p-t^q \le (q-p)/q$ for $0<p\le q$ and $0\le t \le 1$.

In order to verify (\ref{2.8}) in Assumption \ref{A4},
in the definition of $\Gamma_\a$ we pick $R>1$ and $0<\phi_0<\pi/2$
such that $R<2\cos\phi_0$. Note that
$$
\varphi_\a(z)=\frac{1-(1-z)^{[1/\a]}}{z}-\frac{1}{\a+z}=\frac{\a-(\a+z)(1-z)^{[1/\a]}}{z(\a+z)}.
$$
By using the fact $(1+\a)^{1/\a}\le e$ we can see
$$
|\varphi_\a(z)|\lesssim  \a^{-1} (1+\a/2)^{1/\a}\lesssim \a^{-1} \quad \mbox{on }  \Gamma_\a^{(1)}.
$$
According to the choice of $R$ and $\phi_0$, we have $1+R^2-2R \cos\phi_0<1$. Thus
$$
|\varphi_\a(z)|\lesssim \frac{\a+(\a+R) (1+R^2-2 R\cos \phi_0)^{[1/\a]/2}}{R(R+\a)} \lesssim 1 \quad
\mbox{on } \Gamma_{\a}^{(2)}.
$$
Furthermore, on  $\Gamma_\a^{(3)}\cup \Gamma_\a^{(4)}$ we have
$$
|\varphi_\a(z)|\lesssim \frac{\a+(\a+t)(1+t^2-2 t\cos \phi_0)^{1/(2\a)}}{t(\a+t)}.
$$
Therefore
\begin{align*}
\int_{\Gamma_\a} |\varphi_\a(z)| |dz| &\lesssim 1+ \int_{\a/2}^R
\frac{\a+(\a+t)(1+t^2-2 t\cos \phi_0)^{1/(2\a)}}{t(\a+t)} dt\\
& = 1+ \int_{1/2}^{R/\a} \frac{1+(1+t)(1+\a^2 t^2 -2\a t \cos\phi_0)^{1/(2\a)}}{t(1+t)} dt\\
&\lesssim 1+\int_{1/2}^{R/\a} (1+\a^2 t^2 -2\a t\cos\phi_0)^{1/(2\a)} dt.
\end{align*}
Observe that for $1/2\le t\le R/\a$ there holds
$$
(1+\a^2 t^2 -2\a t\cos\phi_0)^{1/(2\a)}\le (1-\mu_0 \a t)^{1/(2\a)}\le e^{-\mu_0 t/2}
$$
with $\mu_0:=2\cos \phi_0-R>0$.  Thus
$$
\int_{\Gamma_\a} |\varphi_\a(z)| |d z| \lesssim 1+\int_{1/2}^\infty e^{-\mu_0 t/2} dt\lesssim 1.
$$
\end{example}

\begin{example}
We consider for $0<\a\le 1$ the function $g_\alpha$ given by
$$
g_\alpha(\lambda)=\sum_{i=1}^{[1/\alpha]}(1+\lambda)^{-i}=\frac{1-(1+\la)^{-[1/\a]}}{\la}
$$
which arises from the Lardy method for linear inverse problems. Then the method (\ref{2})
becomes
\begin{align*}
u_{n,0}&=x_n,\\
u_{n, l+1}&=u_{n,l}-(L^{2s}+T_n^* T_n)^{-1} T_n^* \left(F(x_n)-y^\d -T_n(x_n-u_{n,l})\right),\\
&\qquad\qquad\qquad\qquad\qquad\qquad\qquad\qquad l=0, \cdots, [1/\a_n]-1,\\
x_{n+1}&= u_{n, [1/\a_n]},
\end{align*}
where $T_n=F'(x_n)$. The residual function is
$r_\a(\la)=(1+\la)^{-[1/\a]}$. Assumption \ref{A4} and Assumption \ref{A5}
can be verified  similarly as in Example \ref{E3} when the sequence
$\{\a_n\}$ is given by $\a_n=1/k_n$ for some integers $k_n\ge 1$.
\end{example}

%\section{\bf Numerical results}



\end{document}